\documentclass{amsart}
\usepackage[utf8]{inputenc}
\usepackage{amssymb}
\usepackage{hyperref}
\usepackage[final]{showkeys} 

\input xy
\xyoption{all}

\theoremstyle{definition}

\newtheorem{mydef}{Definition}[section]
\newtheorem{abstractdef}{Definition}
\newtheorem{athm}[abstractdef]{Theorem}
\newtheorem{lem}[mydef]{Lemma}
\newtheorem{thm}[mydef]{Theorem}
\newtheorem{conjecture}[mydef]{Conjecture}

\newtheorem{hypothesis}[mydef]{Hypothesis}
\newtheorem{prop}[mydef]{Proposition}
\newtheorem{defin}[mydef]{Definition}

\newtheorem{remark}[mydef]{Remark}
\newtheorem{notation}[mydef]{Notation}
\newtheorem{fact}[mydef]{Fact}

\newcommand{\fct}[2]{{}^{#1}#2}

\newcommand{\ba}{\bar{a}}
\newcommand{\bb}{\bar{b}}

\newcommand{\bx}{\bar{x}}
\newcommand{\by}{\bar{y}}

\newcommand{\Ksatpp}[2]{{#1}^{#2\text{-sat}}}
\newcommand{\Ksatp}[1]{\Ksatpp{\K}{#1}}

\newcommand{\sea}{\mathfrak{C}}

\newcommand{\cf}[1]{\text{cf} (#1)}
\newcommand{\seq}[1]{\langle #1 \rangle}
\newcommand{\rest}{\upharpoonright}

\newcommand{\s}{\mathfrak{s}}

\newcommand{\wkperp}{\underset{\text{wk}}{\perp}}

\newcommand{\K}{\mathcal{K}}

\def\lea{\le}

\def\gea{\ge}

\newbox\noforkbox \newdimen\forklinewidth
\forklinewidth=0.3pt \setbox0\hbox{$\textstyle\smile$}
\setbox1\hbox to \wd0{\hfil\vrule width \forklinewidth depth-2pt
 height 10pt \hfil}
\wd1=0 cm \setbox\noforkbox\hbox{\lower 2pt\box1\lower
2pt\box0\relax}
\def\unionstick{\mathop{\copy\noforkbox}\limits}
\newcommand{\nf}{\unionstick}

\setbox0\hbox{$\textstyle\smile$}
\setbox1\hbox to \wd0{\hfil{\sl /\/}\hfil} \setbox2\hbox to
\wd0{\hfil\vrule height 10pt depth -2pt width
               \forklinewidth\hfil}
\newbox\doesforkbox
\setbox\doesforkbox\hbox{\lower 0pt\box1 \lower
2pt\box2\lower2pt\box0\relax}

\def\1nf{\unionstick^{(1)}}

\def\2nf{\unionstick^{(2)}}
\def\3nf{\unionstick^{(3)}}

\newcommand{\tp}{\text{tp}}
\newcommand{\gtp}{\text{gtp}}

\newcommand{\gS}{\text{gS}}

\newcommand{\Sbs}{\gS^\text{bs}}

\newcommand{\hanf}[1]{h (#1)}

\newcommand{\goodp}{\text{good}^+}

\newcommand{\LS}{\operatorname{LS}}

\newcommand{\negp}{\neg^\ast p}
\newcommand{\Kneg}{\K_{\negp}}

\newcommand{\Av}{\text{Av}}

\title[Categoricity in AECs with primes]{Shelah's eventual categoricity conjecture in tame AECs with primes}
\date{\today}

\author{Sebastien Vasey}
\email{sebv@cmu.edu}
\urladdr{http://math.cmu.edu/\textasciitilde svasey/}
\address{Department of Mathematical Sciences, Carnegie Mellon University, Pittsburgh, Pennsylvania, USA}
\thanks{This material is based upon work done while the author was supported by the Swiss National Science Foundation under Grant No.\ 155136.}

\begin{document}

\begin{abstract}

A new case of Shelah's eventual categoricity conjecture is established:

\begin{athm}\label{abstract-thm}
  Let $\K$ be an AEC with amalgamation. Write $H_2 := \beth_{\left(2^{\beth_{\left(2^{\LS (\K)}\right)^+}}\right)^+}$. Assume that $\K$ is $H_2$-tame and $\K_{\ge H_2}$ has primes over sets of the form $M \cup \{a\}$. If $\K$ is categorical in some $\lambda > H_2$, then $\K$ is categorical in all $\lambda' \ge H_2$.
\end{athm}

The result had previously been established when the stronger locality assumptions of full tameness and shortness are also required.

An application of the method of proof of Theorem \ref{abstract-thm} is that Shelah's categoricity conjecture holds in the context of homogeneous model theory (this was known, but our proof gives new cases):

\begin{athm}\label{abstract-thm-2}
  Let $D$ be a homogeneous diagram in a first-order theory $T$. If $D$ is categorical in a $\lambda > |T|$, then $D$ is categorical in all $\lambda' \ge \min (\lambda, \beth_{(2^{|T|})^+})$.
\end{athm} 

\end{abstract}

\maketitle

\tableofcontents

\section{Introduction}

Shelah's eventual categoricity conjecture is a major force in the development of classification theory for abstract elementary classes (AECs)\footnote{For a history, see the introduction of \cite{ap-universal-v10}. We assume here that the reader is familiar with the basics of AECs as presented in e.g.\ \cite{baldwinbook09}.}.

\begin{conjecture}[Shelah's eventual categoricity conjecture, N.4.2 in \cite{shelahaecbook}]
  An AEC categorical in a high-enough cardinal is categorical on a tail of cardinals.
\end{conjecture} 

In \cite{ap-universal-v10}, we established the conjecture for universal classes with the amalgamation property\footnote{After the initial submission of this paper, we managed to remove the amalgamation hypothesis \cite{categ-universal-2-v3-toappear}.} (a universal class is a class of models closed under isomorphisms, substructures, and unions of $\subseteq$-increasing chains, see \cite{sh300-orig}). The proof starts by noting that universal classes satisfy tameness: a locality property introduced in VanDieren's 2002 Ph.D.\ thesis (the relevant chapter appears in \cite{tamenessone}).

\begin{fact}[\cite{tameness-groups}]\label{univ-tame} Any universal class $\K$ is\footnote{While the main idea of the proof is due to Will Boney, the fact that it applies to universal classes is due to the author. A full proof of Fact \ref{univ-tame} appears as \cite[3.7]{ap-universal-v10}.}  $\LS (\K)$-tame.
\end{fact}

The proof generalizes to give a stronger locality property introduced in \cite{tamelc-jsl}:

\begin{defin}\label{shortness-def}
  Let $\K$ be an AEC and let $\chi \ge \LS (\K)$ be an infinite cardinal. $\K$ is \emph{fully $\chi$-tame and short} if for any $M \in \K$, any ordinal $\alpha$, and any Galois types $p, q \in \gS^\alpha (M)$ of length $\alpha$, $p = q$ if and only if $p^I \rest M_0 = q^I \rest M_0$ for any $M_0 \in \K_{\le \chi}$ with $M_0 \lea M$ and any $I \subseteq \alpha$ with $|I| \le \chi$.
\end{defin}

\begin{fact}
  Any universal class $\K$ is fully $\LS (\K)$-tame and short.
\end{fact}

Another important property of universal classes used in the proof of Shelah's eventual categoricity conjecture \cite[5.20]{ap-universal-v10} is that they have primes. The definition is due to Shelah and appears in \cite[III.3]{shelahaecbook}. For the convenience of the reader, we include it here:

\begin{defin}\label{prime-def}
  Let $\K$ be an AEC.

  \begin{enumerate}
  \item\label{represents} We say a triple $(a, M, N)$ \emph{represents} a Galois type $p$ if $p = \gtp (a / M; N)$. In particular, $M \lea N$ and $a \in |N|$.
  \item A \emph{prime triple} is a triple $(a, M, N)$ representing a nonalgebraic Galois type $p$ such that for every $N' \in \K$, $a' \in |N'|$, if $p = \gtp (a' / M; N')$ then there exists $f: N \xrightarrow[M]{} N'$ so that $f (a) = a'$.
    \item We say that $\K$ \emph{has primes} if for every $M \in \K$ and every nonalgebraic $p \in \gS (M)$, there exists a prime triple representing $p$.
    \item We define localizations such as ``$\K_\lambda$ has primes'' in the natural way.
  \end{enumerate}
\end{defin}

By taking the closure of $|M| \cup \{a\}$ under the functions of $N$, we get:

\begin{fact}[5.3 in \cite{ap-universal-v10}]
  Any universal class has primes.
\end{fact}

The proof of the eventual categoricity conjecture for universal classes with amalgamation in \cite{ap-universal-v10} generalizes to give: 

\begin{fact}[5.18 in \cite{ap-universal-v10}]\label{shelah-conjecture-fact}
  Fully tame and short AECs that have amalgamation and primes satisfy Shelah's eventual categoricity conjecture.
\end{fact}

Many results only use the assumption of tameness (for example \cite{tamenessone, tamenesstwo, tamenessthree, b-k-vd-spectrum, lieberman2011, ss-tame-jsl, bv-sat-v3}), while others use full tameness and shortness \cite{bg-v11-toappear, indep-aec-apal} (but it is also unclear whether it is really needed there, see \cite[Question 15.4]{indep-aec-apal}). 

It is natural to ask whether shortness can be removed from Fact \ref{shelah-conjecture-fact}. We answer in the affirmative: Tame AECs with primes and amalgamation satisfy Shelah's eventual categoricity conjecture. To state this more precisely, we adopt notation from \cite[Chapter 14]{baldwinbook09}.

\begin{notation}\label{h-notation}
  For $\lambda$ an infinite cardinal, let $\hanf{\lambda} := \beth_{(2^{\lambda})^+}$. For $\K$ a fixed AEC, write $H_1 := \hanf{\LS (\K)}$ and $H_2 := \hanf{H_1} = \hanf{\hanf{\LS (\K)}}$.
\end{notation}

\textbf{Main Theorem \ref{main-thm-proof}.}
Let $\K$ be an AEC with amalgamation. Assume that $\K$ is $H_2$-tame and $\K_{\ge H_2}$ has primes. If $\K$ is categorical in some $\lambda > H_2$, then $\K$ is categorical in all $\lambda' \ge H_2$.

This improves \cite[5.18]{ap-universal-v10} which assumed full $\LS (\K)$-tameness and shortness (so the improvement is on two counts: ``full tameness and shortness'' is replaced by ``tameness'' and ``$\LS (\K)$'' is replaced by ``$H_2$''). Compared to Grossberg and VanDieren's upward transfer \cite{tamenessthree}, we do \emph{not} require categoricity in a successor cardinal, but we \emph{do} require the categoricity cardinal to be at least $H_2$ and more importantly ask for the AEC to have primes.

Let us give a rough picture of the proof of both Theorem \ref{main-thm-proof} and the earlier \cite[5.18]{ap-universal-v10}. We will then explain where exactly the two proofs differ. The first step of the proof is to find a sub-AEC $\K'$ of $\K$ (typically a class of saturated models or just a tail: in the case of Theorem \ref{main-thm-proof} we will have $\K' = \K_{\ge H_2}$) which is ``well-behaved'' in the sense of admitting a good-enough notion of independence. Typically, the first step does not use primes. The second step is to show that in $\K'$, categoricity in \emph{some} $\lambda > \LS (\K')$ implies categoricity in \emph{all} $\lambda' > \LS (\K')$. This uses orthogonality calculus and the existence of prime models. The third step pulls back this categoricity transfer to $\K$.

Shelah has developed orthogonality calculus in the context of what he calls successful $\goodp$ $\lambda$-frames \cite[III.6]{shelahaecbook}. It is known \cite{indep-aec-apal} that one can build such a frame using categoricity, amalgamation, and full tameness and shortness so this is how $\K'$ from the previous paragraph was chosen in \cite{ap-universal-v10}. The orthogonality calculus part was just quoted from Shelah (although we did provide some proofs for the convenience of the reader). It is not known how to build a successful $\goodp$ $\lambda$-frame using just categoricity, amalgamation, and \emph{tameness}.

In this paper, we develop orthogonality calculus in the setup of good $\lambda$-frames with primes (i.e.\ we get rid of the successful $\goodp$ hypothesis). Note that it is easier to build good frames than to build successful ones (see \cite{ss-tame-jsl} and \cite[6.14]{vv-symmetry-transfer-v3}). In particular, this can be done with just amalgamation, categoricity, and tameness (the threshold cardinals are also lower than in the construction of a successful good frame).

To develop orthogonality calculus in good frames with primes, we change Shelah's definition of orthogonality: Shelah's definition uses the so-called uniqueness triples, which may not exist here. This paper's definition uses prime triples instead and shows that the proofs needed for the categoricity transfer still go through. This is the main difference between this paper and \cite{ap-universal-v10}. In some places, new arguments are provided. For example, Lemma \ref{perp-exist}, saying that a definition of orthogonality in terms of ``for all'' is equivalent to one in terms of ``there exists'', has a different proof than Shelah's.

Let us justify the assumptions of Theorem \ref{main-thm-proof}. First of all, why do we ask for $\lambda > H_2$ and not e.g.\ $\lambda > H_1$ or even $\lambda > \LS (\K)$? The reason is that the argument uses categoricity in \emph{two} cardinals, so we appeal to a downward categoricity transfer implicit in \cite[II.1.6]{sh394} which proves (without using primes) that classes as in the hypothesis of Theorem \ref{main-thm-proof} must be categorical in $H_2$. If we know that the class is categorical in two cardinals already, then we can work above $\LS (\K)$ (provided of course we adjust the levels at which tameness and primes occur). This is Theorem \ref{upward-transfer}. Moreover if we know that for some $\chi < \lambda$, the class of $\chi$-saturated models of $\K$ has primes, then we can also lower the Hanf number from $H_2$ to $H_1$ (see Theorem \ref{main-thm-sat}).

Let us now discuss the structural assumptions on $\K$. Many classes occurring in practice have amalgamation. Grossberg conjectured \cite[2.3]{grossberg2002} that eventual amalgamation should follow from categoricity and, assuming that the class is \emph{eventually syntactically characterizable} (see \cite[Section 4]{ap-universal-v10}), it does assuming the other assumptions: tameness and having primes. We now focus on these two assumptions.

A wide variety of AECs are tame (see e.g.\ the introduction to \cite{tamenessone} or the upcoming survey \cite{bv-survey-v4-toappear}), and many classes studied by algebraists have primes (one example are AECs which admit intersections, i.e.\ whenever $N \in \K$ and $A \subseteq |N|$, we have that $\bigcap \{M \lea N \mid A \subseteq |M|\} \lea N$. See \cite{non-locality} or \cite[Section 2]{ap-universal-v10}). Tameness is conjectured (see \cite[Conjecture 1.5]{tamenessthree}) to follow from categoricity and of course, the existence of prime models plays a key role in many categoricity transfer results including Morley's categoricity theorem and Shelah's generalization to excellent classes \cite{sh87a, sh87b}. Currently, no general way\footnote{We discuss homogeneous model theory and more generally finitary AECs later.}  of building prime models in AECs is known except by going through the machinery of excellence \cite[Chapter III]{shelahaecbook}. It is unknown whether excellence follows from categoricity.

In the special case of homogeneous model theory, it is easier to build prime models\footnote{We thank Rami Grossberg for asking us if the methods of \cite{ap-universal-v10} could be adapted to this context.}. Let $\K$ be a class of models of a homogeneous diagram categorical in a $\lambda > H_2$. Clearly, $\K$ has amalgamation and is fully $\LS (\K)$-tame and short. By stability and \cite[Section 5]{sh3}, the class of $H_2$-saturated models of $\K$ has primes. The proof of Theorem \ref{main-thm-proof} first argues without using primes that $\K$ is categorical in $H_2$. Hence the class of $H_2$-saturated models of $\K$ is just the class $\K_{\ge H_2}$, so it has primes. We apply Theorem \ref{main-thm-proof} to obtain the eventual categoricity conjecture for homogeneous model theory. Actually Theorem \ref{main-thm-proof} is not needed for that result: \cite[5.18]{ap-universal-v10} suffices. However we can also improve on the Hanf number $H_2$ and obtain Theorem \ref{abstract-thm-2} from the abstract:

\textbf{Theorem \ref{abstract-thm-2-proof}.}
  Let $D$ be a homogeneous diagram in a first-order theory $T$. If $D$ is categorical in some $\lambda > |T|$, then $D$ is categorical in all $\lambda' \ge \min (\lambda, \hanf{|T|})$.

  When $T$ is countable, a stronger result has been established by Lessmann \cite{lessmann2000}: categoricity in some uncountable cardinal implies categoricity in \emph{all} uncountable cardinals. When $T$ is uncountable, the eventual categoricity conjecture for homogeneous model theory is implicit in \cite[Section 7]{sh3} and was also given a proof by Hyttinen \cite{hyt-categ-homog}. More precisely, Hyttinen prove that categoricity in some $\lambda > |T|$ with $\lambda \neq \aleph_{\omega} (|T|)$ implies categoricity in all $\lambda' \ge \min (\lambda, \hanf{|T|})$.  Our proof of Theorem \ref{abstract-thm-2-proof} is new and also covers the case $\lambda = \aleph_{\omega} (|T|)$. We do not know whether a similar result also holds in the framework of finitary AECs (there the categoricity conjecture has been solved for tame and \emph{simple}\footnote{In this context, stable does not imply simple.} finitary AECs with countable Löwenheim-Skolem number \cite{finitary-aec}\footnote{The argument is similar to the proof of Morley's categoricity theorem.}).

A continuation of the present paper is in \cite{downward-categ-tame-apal} (circulated after the initial submission of this paper), where orthogonality calculus is developed inside good frames that do not necessarily have primes. We establish there that the analog of Theorem \ref{abstract-thm-2-proof} (i.e.\ the threshold is $H_1$) holds in any $\LS (\K)$-tame AEC with amalgamation and primes.

This paper was written while the author was working on a Ph.D.\ thesis under the direction of Rami Grossberg at Carnegie Mellon University and he would like to thank Professor Grossberg for his guidance and assistance in his research in general and in this work specifically. The author also thanks Tapani Hyttinen for his comments on the categoricity conjecture for homogeneous model theory, as well as the referee for several thorough reports that greatly helped improve the presentation and focus of this paper.

\section{Orthogonality with primes}

In \cite[III.6]{shelahaecbook}, Shelah develops a theory of orthogonality for good frames. In addition to the existence of primes, his assumptions include that the good frame is successful $\goodp$ (see \cite[III.1]{shelahaecbook}) so in particular it expands to an independence relation NF for models in $\K_\lambda$. While successfulness follows from full tameness and shortness \cite[11.13]{indep-aec-apal}, it is not clear if it follows from tameness only, so we do not adopt this assumption. Instead we will assume only that the good frame has primes.\footnote{Recently, Will Boney and the author have shown \cite{counterexample-frame-v2} that the $\aleph_{n - 3}$-good frame in the Hart-Shelah example is \emph{not} (weakly) successful. However it is categorical and has primes (because the Hart-Shelah example admits intersections). Thus the setup of this paper is strictly weaker than Shelah's.} 

The proof of Fact \ref{shelah-conjecture-fact} uses Shelah's theory of orthogonality to prove a technical statement on good frames being preserved when doing a certain change of AEC \cite[III.12.39]{shelahaecbook}. We show that this statements still holds if we do \emph{not} assume successfulness but only the existence of primes (see Theorem \ref{not-unidim-frame-1}). Along the way, we develop orthogonality calculus in good frames with primes. To do so, we change Shelah's definition of orthogonality from \cite[III.6.2]{shelahaecbook} to use prime triples instead of uniqueness triples and check that \cite[III.12.39]{shelahaecbook} can still be proven using this new definition of orthogonality.

We assume that the reader is familiar with Section 5 and Appendix B of \cite{ap-universal-v10}. We also assume that the reader is familiar with the basics of good frames as presented in \cite[II.2]{shelahaecbook}. As in \cite[II.6.35]{shelahaecbook}, we say that a good $\lambda$-frame $\s$ is \emph{type-full} if the basic types consist of \emph{all} the nonalgebraic types over $M$. For simplicity, we focus on type-full good frames here. We say that a good $\lambda$-frame $\s$ is \emph{on $\K_\lambda$} if its underlying class is $\K_\lambda$. We say that $\s$ is \emph{categorical} if $\K$ is categorical in $\lambda$ and we say that it \emph{has primes} if $\K_\lambda$ has primes (where we localize Definition \ref{prime-def} in the natural way).

All throughout, we assume:

\begin{hypothesis}\label{good-frame-hyp}
  $\s = (\K_\lambda, \nf, \Sbs)$ is a categorical type-full good $\lambda$-frame which has primes. We work inside $\s$.
\end{hypothesis}

Hypothesis \ref{good-frame-hyp} is reasonable: By Fact \ref{good-frame-fact}, categorical good frames exist assuming categoricity, amalgamation, and tameness. As for assuming the existence of primes, this is an hypothesis of our main theorem (Theorem \ref{main-thm-proof}) and we have tried to justify it in the introduction. Se also Fact \ref{homog-primes}, which shows how to obtain the existence of primes in the setup of homogeneous model theory.

The definition of orthogonality is similar to \cite[III.6.2]{shelahaecbook}: the only difference is that uniqueness triples are replaced by prime triples. In Shelah's context, this gives an equivalent definition (see \cite[III.6.3]{shelahaecbook}).

\begin{defin}\label{perp-def}
  Let $M \in \K_\lambda$ and let $p, q \in \gS (M)$ be nonalgebraic. We say that $p$ is \emph{weakly orthogonal to $q$} and write $p \wkperp q$ if for all prime triples $(b, M, N)$ representing $q$ (i.e.\ $q = \gtp (b / M; N)$, see Definition \ref{prime-def}(\ref{represents})), we have that $p$ has a unique extension to $\gS (N)$.

  We say that \emph{$p$ is orthogonal to $q$} (written $p \perp q$) if for every $N \in \K_\lambda$ with $N \gea M$, $p' \wkperp q'$, where $p'$, $q'$ are the nonforking extensions to $N$ of $p$ and $q$ respectively.

  For $p_\ell \in \gS (M_\ell)$ nonalgebraic, $\ell = 1,2$, $p_1 \perp p_2$ if and only if there exists $N \gea M_\ell$, $\ell = 1,2$ such that the nonforking extensions to $N$ $p_1'$ and $p_2'$ of $p_1$ and $p_2$ respectively are orthogonal. 
\end{defin}
\begin{remark}
  Formally, the definition of orthogonality depends on the frame but $\s$ will always be fixed.
\end{remark}

The next basic lemma says that we can replace the ``for all'' in Definition \ref{perp-def} by ``there exists''. This corresponds to \cite[III.6.3]{shelahaecbook}, but the proof is different.

\begin{lem}\label{perp-exist}
  Let $M \in \K_\lambda$ and $p, q \in \gS (M)$ be nonalgebraic. Then $p \wkperp q$ if and only if there exists a prime triple $(b, M, N)$ representing $q$ such that $p$ has a unique extension to $\gS (N)$.
\end{lem}
\begin{proof}
  The left to right direction is straightforward. Now assume $(b, M, N)$ is a prime triple representing $q$ such that $p$ has a unique extension to $\gS (N)$. Let $(b_2, M, N_2)$ be another prime triple representing $q$. We want to see that $p$ has a unique extension to $\gS (N_2)$. Let $p_2 \in \gS (N_2)$ be an extension of $p$. By primeness of $(b_2, M, N_2)$, there exists $f: N_2 \xrightarrow[M]{} N$ such that $f (b_2) = b$.

  We have that $f (p_2)$ is an element of $\gS (f[N_2])$ and $f[N_2] \lea N$, so using amalgamation pick $p_2' \in \gS (N)$ extending $f (p_2)$. Now as $f$ fixes $M$, $f (p_2)$ extends $p$, so $p_2'$ extends $p$. Since by assumption $p$ has a unique extension to $\gS (N)$, $p_2'$ must be this unique extension, and in particular $p_2'$ does not fork over $M$. By monotonicity, $f (p_2)$ does not fork over $M$. By invariance, $p_2$ does not fork over $M$. This shows that $p_2$ must be the unique extension of $p$ to $\gS (N_2)$, as desired.
\end{proof}

We now show that weak orthogonality is the same as orthogonality. Recall (Hypothesis \ref{good-frame-hyp}) that we are assuming categoricity in $\lambda$. In particular, \emph{all the models of size $\lambda$ are superlimit}\footnote{\label{superlim-footnote}Recall \cite[N.2.4(4)]{shelahaecbook} that $M \in \K_\lambda$ is superlimit if it is universal in $\K_\lambda$, has a proper extension, and whenever $\delta < \lambda^+$ is limit, $\seq{M_i : i < \delta}$ is increasing with $M \cong M_i$ for all $i < \delta$, then $M \cong \bigcup_{i < \delta} M_i$. Directly from the definition, one checks that for any AEC $\K$ and any $\lambda \ge \LS (\K)$, if $\K$ is categorical in $\lambda$ and has no maximal models in $\lambda$ (so in particular if there is a categorical good $\lambda$-frame on $\K_\lambda$), then the model of cardinality $\lambda$ is superlimit.}. Thus we can use the following property, which Shelah proves for superlimit models $M, N \in \K_\lambda$:

\begin{fact}[The conjugation property, III.1.21 in \cite{shelahaecbook}]\label{conj-prop}
  Let $M \lea N$ be in $\K_\lambda$, $\alpha < \lambda$, and let $(p_i)_{i < \alpha}$ be types in $\gS (N)$ that do not fork over $M$. Then there exists $f: N \cong M$ such that $f (p_i) = p_i \rest M$ for all $i < \alpha$.
\end{fact}

\begin{lem}[III.6.8(5) in \cite{shelahaecbook}]\label{perp-weak}
  For $M \in \K_\lambda$, $p, q \in \gS (M)$ nonalgebraic, $p \wkperp q$ if and only if $p \perp q$.
\end{lem}
\begin{proof}
  Clearly if $p \perp q$ then $p \wkperp q$. Conversely assume $p \wkperp q$ and let $N \gea M$. Let $p', q'$ be the nonforking extensions to $N$ of $p$, $q$ respectively. We want to show that $p' \wkperp q'$. By the conjugation property, there exists $f: N \cong M$ such that $f (p') = p$ and $f (q) = q'$. Since weak orthogonality is invariant under isomorphism, $p' \wkperp q'$.
\end{proof}

We have arrived to the main theorem of this section. This generalizes \cite[III.12.39]{shelahaecbook} (a full proof of which appears in \cite[B.7]{ap-universal-v10}) which assumes in addition that $\s$ is successful and $\goodp$. For the convenience of the reader, we repeat Hypothesis \ref{good-frame-hyp}.

\begin{thm}\label{not-unidim-frame-1}
  Let $\s = (\K_\lambda, \nf, \Sbs)$ be a categorical good $\lambda$-frame which has primes. If $\K_\lambda$ is not weakly uni-dimensional (see \cite[III.2.2(6)]{shelahaecbook}), then there exists $M \in \K_\lambda$ and $p \in \gS (M)$ such that $\s \rest \Kneg$ (the expansion of $\s$ to $\K_M$ restricted to the models in $\Kneg$, see \cite[2.20, 5.7]{ap-universal-v10}) is a type-full good $\lambda$-frame with primes.
\end{thm}
\begin{proof}
  Exactly the same as in \cite[B.7]{ap-universal-v10}, except that we replace uniqueness triples with prime triples, and use Lemmas \ref{perp-exist} and \ref{perp-weak} wherever appropriate.
\end{proof}

Assuming tameness and existence of primes above $\lambda$, we can conclude an equivalence between uni-dimensionality and categoricity. Once again, we repeat Hypothesis \ref{good-frame-hyp}.

\begin{thm}\label{unidim-equiv}
  Assume that $\K$ is an AEC categorical in $\lambda$ which has a (type-full) good $\lambda$-frame. If $\K_{\ge \lambda}$ has primes and is $\lambda$-tame, then the following are equivalent:

  \begin{enumerate}
    \item $\K$ is weakly uni-dimensional (see \cite[III.2.2(6)]{shelahaecbook}).
    \item $\K$ is categorical in \emph{all} $\mu > \lambda$.
    \item $\K$ is categorical in \emph{some} $\mu > \lambda$.
  \end{enumerate}
\end{thm}
\begin{proof}
  Exactly as in the proof of \cite[5.16]{ap-universal-v10}, except that we use Theorem \ref{not-unidim-frame-1} (and replace uniqueness triples with prime triples).
\end{proof}
\begin{remark}
  For the proof of Theorem \ref{unidim-equiv} (and the other categoricity transfer theorems of this paper), the symmetry property of good frames is not needed.
\end{remark}

\section{Categoricity transfers in AECs with primes}

In this section, we prove Theorem \ref{abstract-thm} from the abstract. We first recall that the existence of good frames follow from categoricity, amalgamation, and tameness. We use the following notation:

\begin{notation}
  For $\K$ an AEC with amalgamation and $\lambda > \LS (\K)$, we write $\Ksatp{\lambda}$ for the class of $\lambda$-saturated models in $\K_{\ge \lambda}$.
\end{notation}


\begin{fact}\label{good-frame-fact}
  Let $\K$ be a $\LS (\K)$-tame AEC with amalgamation and no maximal models. Let $\lambda$ and $\mu$ be cardinals such that both $\lambda$ and $\mu$ are strictly bigger than $\LS (\K)$. If $\K$ is categorical in $\mu$, then:
  \begin{enumerate}
    \item $\K$ is stable in every cardinal.
    \item $\Ksatp{\lambda}$ is an AEC with $\LS (\Ksatp{\lambda}) = \lambda$.
    \item There exists a categorical type-full good $\lambda$-frame with underlying class $\Ksatp{\lambda}_\lambda$.
  \end{enumerate}
\end{fact}
\begin{proof}
  By the Shelah-Villaveces theorem \cite[2.2.1]{shvi635} (see \cite[5.3]{gv-superstability-v4} for a statement of the version with full amalgamation and the recent \cite{shvi-notes-v3-toappear} for a detailed proof), $\K$ is $\LS (\K)$-superstable (see for example \cite[10.1]{indep-aec-apal}), in particular it is stable in $\LS (\K)$. Now we start to use $\LS (\K)$-tameness. By \cite[6.10]{vv-symmetry-transfer-v3}, $\Ksatp{\lambda}$ is an AEC with $\LS (\Ksatp{\lambda}) = \lambda$. By \cite[10.8]{indep-aec-apal}, there is a type-full good $\lambda$-frame with underlying class $\Ksatp{\lambda}_{\lambda}$ (and in particular stable in $\lambda$) By uniqueness of saturated models, $\Ksatp{\lambda}$ is categorical in $\lambda$.
\end{proof}


We obtain a categoricity transfer for tame AECs with primes categorical in two cardinals. First we prove a more general lemma:

\begin{lem}\label{upward-transfer-lem}
  Let $\K$ be a $\LS (\K)$-tame AEC with amalgamation and arbitrarily large models. Let $\lambda$ and $\mu$ be cardinals such that $\LS (\K) < \lambda < \mu$. 

  If $\K$ is categorical in $\mu$ and $\Ksatp{\lambda}$ has primes, then $\Ksatp{\lambda}$ is categorical in all $\mu' \ge \lambda$.
\end{lem}
\begin{proof}
  By partitioning $\K$ into disjoint AECs, each of which has joint embedding (see for example \cite[16.14]{baldwinbook09}) and working inside the unique piece that is categorical in $\mu$, we can assume without loss of generality that $\K$ has joint embedding. Because $\K$ has arbitrarily large models, $\K$ also has no maximal models.
  
  By Fact \ref{good-frame-fact}, there is a categorical type-full good $\lambda$-frame $\s$ with underlying class $\Ksatp{\lambda}_\lambda$. Now apply Theorem \ref{unidim-equiv} to $\s$ and $\Ksatp{\lambda}$.
\end{proof}

\begin{thm}\label{upward-transfer}
  Let $\K$ be a $\LS (\K)$-tame AEC with amalgamation and arbitrarily large models. Let $\lambda$ and $\mu$ be cardinals such that $\LS (\K) < \lambda < \mu$. Assume that $\K_{\ge \lambda}$ has primes.

  If $\K$ is categorical in both $\lambda$ and $\mu$, then $\K$ is categorical in all $\mu' \ge \lambda$.
\end{thm}
\begin{proof}
  By categoricity, $\Ksatp{\lambda} = \K_{\ge \lambda}$. Now apply Lemma \ref{upward-transfer-lem}.
\end{proof}
\begin{remark}
  What if $\lambda = \LS (\K)$? Then it is open whether $\K$ has a good $\LS (\K)$-frame (see the discussion in \cite[Section 3]{ss-tame-jsl}). If it does, then we can use Theorem \ref{unidim-equiv}.
\end{remark}

We present two transfers from categoricity in a single cardinal. The first uses the following downward transfer which follows from the proof of \cite[14.9]{baldwinbook09} (an exposition of \cite[II.1.6]{sh394}).

\begin{fact}\label{downward-transfer}
Let $\K$ be an AEC with amalgamation and no maximal models. If $\K$ is categorical in a $\lambda > H_2$ (recall Notation \ref{h-notation}) and the model of size $\lambda$ is $H_2^+$-saturated, then $\K$ is categorical in $H_2$.
\end{fact}

To get the optimal tameness bound, we will use\footnote{For a simpler proof of Theorem \ref{main-thm-proof} from slightly stronger assumptions, replace ``$H_2$-tame'' by ``$\chi$-tame for some $\chi < H_2$. Then in the proof one can use Fact \ref{good-frame-fact} together with Theorem \ref{upward-transfer}, both applied to the class $\K_{\ge \chi}$.}:

\begin{fact}[7.9 in \cite{vv-symmetry-transfer-v3}]\label{good-frame-fact-2}
  Let $\K$ be an AEC with amalgamation and no maximal models. Let $\mu \ge H_1$ and assume that $\K$ is categorical in a $\lambda > \mu$ so that the model of size $\lambda$ is $\mu^+$-saturated. Then there exists a categorical type-full good $\mu$-frame with underlying class $\Ksatp{\mu}_{\mu}$.
\end{fact}

\begin{thm}\label{main-thm-proof}
  Let $\K$ be an AEC with amalgamation. Assume that $\K$ is $H_2$-tame and $\K_{\ge H_2}$ has primes. If $\K$ is categorical in some $\lambda > H_2$, then $\K$ is categorical in all $\lambda' \ge H_2$.
\end{thm}
\begin{proof}
  As in the proof of Lemma \ref{upward-transfer-lem}, we can assume without loss of generality that $\K$ has no maximal models. By Fact \ref{good-frame-fact} (applied to the AEC $\K_{\ge H_2}$), $\K$ is in particular stable in $\lambda$, hence the model of size $\lambda$ is saturated. By Fact \ref{downward-transfer}, $\K$ is categorical in $H_2$. By Fact \ref{good-frame-fact-2}, there is a categorical type-full good $H_2$-frame $\s$ with underlying class $\Ksatp{H_2}_{H_2}$. By categoricity in $H_2$, $\Ksatp{H_2} = \K_{\ge H_2}$. Now apply Theorem \ref{unidim-equiv} to $\s$.
\end{proof}

We give a variation on Theorem \ref{main-thm-proof} which gives a lower Hanf number but assumes that classes of saturated models have primes. We will use the following consequence of the omitting type theorem for AECs \cite[II.1.10]{sh394} (or see \cite[14.3]{baldwinbook09}):

\begin{fact}\label{omit-type}
  Let $\K$ be an AEC with amalgamation. Let $\lambda \ge \chi > \LS (\K)$ be cardinals. Assume that all the models of size $\lambda$ are $\chi$-saturated. Then all the models of size at least $\min (\lambda, \sup_{\chi_0 < \chi} \hanf{\chi_0})$ are $\chi$-saturated.
\end{fact}

\begin{thm}\label{main-thm-sat}
  Let $\K$ be a $\LS (\K)$-tame AEC with amalgamation and arbitrarily large models. Let $\lambda > \LS (\K)^+$ be such that $\K$ is categorical in $\lambda$ and let $\chi \in (\LS (\K), \lambda)$ be such that $\Ksatp{\chi}$ has primes. Then $\K$ is categorical in all $\lambda' \ge \min (\lambda, \sup_{\chi_0 < \chi} \hanf{\chi_0})$.
\end{thm}
\begin{proof}
  As in the proof of Lemma \ref{upward-transfer-lem}, we may assume that $\K$ has no maximal models. By Lemma \ref{upward-transfer-lem}, $\Ksatp{\chi}$ is categorical in all $\lambda' \ge \chi$. By Fact \ref{good-frame-fact}, $\K$ is stable in $\lambda$, so the model of size $\lambda$ is saturated, hence $\chi$-saturated. By Fact \ref{omit-type}, all the models of size at least $\lambda_0' := \min (\lambda, \sup_{\chi_0 < \chi} \hanf{\chi_0})$ are $\chi$-saturated. In other words, $\K_{\ge \lambda_0'} = \Ksatp{\chi}_{\ge \lambda_0'}$. Since $\Ksatp{\chi}$ is categorical in all $\lambda' \ge \chi$, $\K$ is categorical in all $\lambda' \ge \lambda_0'$.
\end{proof}
\begin{remark}
  Theorem \ref{main-thm-proof} and Theorem \ref{main-thm-sat} have different strengths. It could be that we know our AEC $\K$ has primes but it is unclear that $\Ksatp{\chi}$ has primes for any $\chi$. For example, $\K$ could be a universal class (or more generally an AEC admitting intersections). In this case we can use Theorem \ref{main-thm-proof}. On the other hand we may not know that $\K$ has primes but we could know how to build primes in $\Ksatp{\chi}$ (for example $\K$ could be an elementary class or more generally a class of homogeneous models, see the next section). There Theorem \ref{main-thm-sat} applies. 
\end{remark}

\section{Categoricity in homogeneous model theory}

We use the results of the previous section to obtain Shelah's categoricity conjecture for homogeneous model theory, a nonelementary framework extending classical first-order model theory. It was introduced in \cite{sh3}. The idea is to look at a class of models of a first-order theory omitting a set of types and assume that this class has a very nice (sequentially homogeneous) monster model. We quote from the presentation in \cite{grle-homog} but all the results on homogeneous model theory that we use initially appeared in either \cite{sh3} or \cite{hs-independence}.

The following definitions appear in \cite{grle-homog}. They differ from (but are equivalent to) Shelah's original definitions from \cite{sh3}.

\begin{defin}
  Fix a first-order theory $T$. 

  \begin{enumerate}
  \item A set of $T$-types $D$ is a \emph{diagram in $T$} if it has the form $\{\tp (\ba / \emptyset; M) \mid \ba \in \fct{<\omega}{A}\}$ for a model $M$ of $T$. 
  \item A model $M$ of $T$ is a \emph{$D$-model} if $D(M) := \{\tp (\ba / \emptyset; M) \mid \ba \in \fct{<\omega}{|M|}\} \subseteq D$. 
  \item For $D$ a diagram of $T$, we let $\K_D$ be the class of $D$-models of $T$, ordered with elementary substructure.
  \item For $M$ a model of $T$, we write $S_D^{<\omega} (A; M)$ for the set of types of finite tuples over $A$ which are realized in some $D$-model $N$ with $N \preceq M$.
  \end{enumerate}
\end{defin}

\begin{defin}
  Let $T$ be a first-order theory and $D$ a diagram in $T$. A model $M$ of $T$ is \emph{$(D, \lambda)$-homogeneous} if it is a $D$-model and for every $N \succeq M$, every $A \subseteq |M|$ with $|A| < \lambda$, every $p \in S_D^{<\omega} (A; N)$ is realized in $M$.
\end{defin}

\begin{defin}
  We say a diagram $D$ in $T$ is \emph{homogeneous} if for every $\lambda$ there exists a $(D, \lambda)$-homogeneous model of $T$.
\end{defin}

We are not aware of any source explicitly stating the facts below, but they are straightforward to check, so we omit the proof. They will be used without mention.

\begin{prop}\label{kd-aec}
  For $D$ a homogeneous diagram in $T$:

  \begin{enumerate}
    \item $\K_D$ is an AEC with $\LS (\K_D) = |T|$. 
    \item $\K$ has amalgamation, no maximal models, and is fully $\LS (\K)$-tame and short (in fact syntactic and Galois types coincide).
    \item For $\lambda > |T|$, a $D$-model $M$ is $(D, \lambda)$-homogeneous if and only if $M \in \Ksatp{\lambda}_D$.
  \end{enumerate}
\end{prop}

Note that in this framework it also makes sense to talk about the $|T|$-saturated models, so we let:

\begin{defin}
  Let $\Ksatp{|T|}_D$ be the class of $(D, |T|)$-homogeneous models, ordered by elementary substructure.
\end{defin}

To apply the results of the previous section, we must give conditions under which $\Ksatp{\chi}_D$ has primes. This is implicit in \cite[Section 5]{sh3}:

\begin{fact}\label{homog-primes}
  Let $D$ be a homogeneous diagram in $T$. If $\K_D$ is stable in $\chi \ge \LS (\K)$ then $\Ksatp{\chi}_D$ has primes.
\end{fact} 
\begin{proof}
  By \cite[5.11(1)]{sh3} (with $\mu, \lambda$ there standing for $\chi, \chi$ here; in particular $2^\mu > \lambda$), $D$ satisfies a property Shelah calls $(P, \chi, 1)$ (a form of density of isolated types, see \cite[5.4]{sh3}). By the proof of \cite[5.2(1)]{sh3} and \cite[5.3(1)]{sh3} there, this implies that the class $\Ksatp{\chi}_D$ has primes.
\end{proof}

We immediately obtain:

\begin{thm}\label{abstract-thm-2-proof-0}
  If a homogeneous diagram $D$ in a first-order theory $T$ is categorical in a $\lambda > |T|^+$, then it is categorical in all $\lambda' \ge \min (\lambda, \hanf{|T|})$.
\end{thm}
\begin{proof}
  Note that $\K_D$ is stable in all cardinals by Fact \ref{good-frame-fact}. So we can combine Fact \ref{homog-primes} and Theorem \ref{main-thm-sat}.
\end{proof}

This proves Theorem \ref{abstract-thm-2} in the abstract modulo a small wrinkle: the case $\lambda = |T|^+$. One would like to use the categoricity transfer of Grossberg and VanDieren \cite{tamenessthree} but they assume that $\K$ is categorical in a successor $\lambda > \LS (\K)^+$ since otherwise it is in general unclear whether there is a superlimit (see footnote \ref{superlim-footnote}) in $\LS (\K)$ (one can get around this difficulty if $\LS (\K) = \aleph_0$, see \cite{lessmann-upward-transfer}). However in the case of homogeneous model theory we can show that there \emph{is} a superlimit, completing the proof. The key is that under stability, $(D, |T|)$-homogeneous models are closed under unions of chains. This is claimed without proof by Shelah in \cite[1.15]{sh54}. We give a proof here which imitates the first-order proof of Harnik \cite{harniksat}. Still it seems that a fair amount of forking calculus has to be developed first. All throughout, we assume:

\begin{hypothesis}
  $D$ is a homogeneous diagram in a first-order theory $T$. We work inside a $(D, \bar{\kappa})$-homogeneous model $\sea$ for $\bar{\kappa}$ a very big cardinal. In particular, all sets are assumed to be $D$-sets (see \cite[2.1(2)]{grle-homog}).
\end{hypothesis}

The following can be seen as a first approximation for forking in the homogeneous context. It was used by Shelah to prove the stability spectrum theorem in this framework (see Fact \ref{stab-spectrum}). We will not use the exact definition, only its consequences.

\begin{defin}[4.1 in \cite{sh3}]
  A type $p \in S_D^{<\omega} (A)$ \emph{strongly splits} over $B \subseteq A$ if there exists an indiscernible sequence $\seq{\ba_i : i < \omega}$ over $B$ and a formula $\phi (\bx, \by)$ such that $\phi (\bx, \ba_0) \in p$ and $\neg \phi (\bx, \ba_1) \in p$.
\end{defin}
\begin{defin}
  $\kappa (D)$ is the minimal cardinal $\kappa$ such that for all $A$ and all $p \in S_D^{<\omega} (A)$, there exists $B \subseteq A$ with $|B| < \kappa$ so that $p$ does not strongly split over $B$.
\end{defin}

The following is due to Shelah \cite[4.4]{sh3}. See also \cite[4.11, 4.14, 4.15]{grle-homog}:

\begin{fact}\label{stab-spectrum}
  If $D$ is stable in $\lambda_0 \ge |T|$, then $\kappa (D) < \infty$ and for $\lambda \ge \lambda_0$, $D$ is stable in $\lambda$ if and only if $\lambda = \lambda^{<\kappa (D)}$.
\end{fact}

We can define forking using strong splitting:

\begin{defin}[3.1 in \cite{hs-independence}]\label{forking-def}
  For $A \subseteq B$, $p \in S_D^{<\omega} (B)$ \emph{does not fork over $A$} if there exists $A_0 \subseteq A$ such that:

  \begin{enumerate}
  \item $|A_0| < \kappa (D)$.
  \item For every set $C$, there exists $q \in S_D^{<\omega} (B \cup C)$ such that $q$ extends $p$ and $q$ does not strongly split over $A_0$.
  \end{enumerate}
\end{defin}

Assuming that the base has a certain degree of saturation, forking behaves well:

\begin{fact}\label{forking-facts}
  Assume that $D$ is stable in $\lambda \ge |T|$. Let $M$ be $(D, \lambda)$-homogeneous and let $A \subseteq B \subseteq C$ be sets.

  \begin{enumerate}
  \item (Monotonicity) For $p \in S_D^{<\omega} (C)$, if $p$ does not fork over $A$, then $p \rest B$ does not fork over $A$ and $p$ does not fork over $B$.
  \item (Extension-existence) For any $p \in S_D^{<\omega} (M)$, there exists $q \in S_D^{<\omega} (M \cup B)$ that extends $p$ and does not fork over $M$. Also, $q$ is algebraic if and only if $p$ is. Moreover if $p \in S_D^{<\omega} (M)$ does not strongly split over $A_0 \subseteq |M|$, then $p$ does not fork over $A_0$.
  \item (Uniqueness) If $p, q \in S_D^{<\omega} (M \cup B)$ both do not fork over $M$ and are such that $p \rest M = q \rest M$, then $p = q$.
  \item (Transitivity) For any $p \in S_D^{<\omega} (M \cup B)$, if $p$ does not fork over $M$ and $p \rest M$ does not fork over $A_0 \subseteq |M|$, then $p$ does not fork over $A_0$.
  \item (Symmetry) If $\tp (\bb / M\ba)$ does not fork over $M$, then $\tp (\ba / M \bb)$ does not fork over $M$.
  \item (Local character) For any $p \in S_D^{<\omega} (M)$, there exists $A_0 \subseteq |M|$ such that $|A_0| < \kappa (D)$ and $p$ does not fork over $A_0$. Moreover, for any $\seq{M_i : i < \delta}$ increasing chain of $(D, \lambda)$-homogeneous models, if $p \in S_D^{<\omega} (\bigcup_{i < \delta} M_i)$ and $\cf{\delta} \ge \kappa (D)$, then there exists $i < \delta$ and $A_0 \subseteq |M_i|$ such that $|A_0| < \kappa (D)$ and $p$ does not fork over $A_0$.
  \end{enumerate}
\end{fact}
\begin{proof}
  We use freely that (by \cite[1.9(iv)]{hs-independence}) a $(D, \lambda)$-homogeneous model is an $a$-saturated model in the sense of \cite[1.8(ii)]{hs-independence}. Monotonicity is \cite[3.2.(i)]{hs-independence}, extension-existence is given by \cite[3.2.(iii), (v), (vi)]{hs-independence} and the definitions of $\kappa (D)$ and forking. Uniqueness is \cite[3.4]{hs-independence}, transitivity is \cite[3.5.(iv)]{hs-independence}, and symmetry is \cite[3.6]{hs-independence}. For local character, we prove the moreover part and the first part follows by taking $M_i := M$ for all $i < \delta$. Let $M_\delta := \bigcup_{i < \delta} M_i$. Without loss of generality, $\delta = \cf{\delta} \ge \kappa (D)$. By definition of $\kappa (D)$, there exists $A_0 \subseteq |M_\delta|$ such that $|A_0| < \kappa (D)$ and $p$ does not strongly split over $A_0$. By cofinality consideration, there exists $i < \delta$ such that $A_0 \subseteq |M_i|$. By the moreover part of extension-existence, for all $j \in [i, \delta)$, $p \rest M_j$ does not fork over $A_0$. By \cite[3.5.(i)]{hs-independence}, it follows that $p$ does not fork over $M_i$, and therefore by transitivity over $A_0$.
\end{proof}

We will use the machinery of indiscernibles and averages. Note that by \cite[3.4, 3.12]{grle-homog}, indiscernible sequences are indiscernible sets under stability. We will use this freely. The following directly follows from the definition of strong splitting:

\begin{fact}[5.3 in \cite{grle-homog}]\label{indisc-fact}
  Assume that $D$ is stable. For all infinite indiscernible sequences $I$ over a set $A$ and all elements $b$, there exists $J \subseteq I$ with $|J| < \kappa (D)$ such that $I \backslash J$ is indiscernible over $A \cup \{b\}$.
\end{fact}

\begin{defin}
  For $I$ an indiscernible sequence of cardinality at least $\kappa (D)$, let $\Av (I / A)$ be the set of formulas $\phi (\bx, \ba)$ with $\ba \in \fct{<\omega}{A}$ such that for at least $\kappa (D)$-many elements $\bb$ of $I$, $\models \phi[\bb, \ba]$.
\end{defin}

\begin{fact}[5.5 in \cite{grle-homog}]
  If $D$ is stable and $I$ is an indiscernible sequence of cardinality at least $\kappa (D)$, then $\Av (I / A) \in S_D^{<\omega} (A)$.
\end{fact}

\begin{fact}\label{indisc-existence}
  Assume that $D$ is stable.
  
  Let $A \subseteq B$ and let $p \in S_D^{<\omega} (B)$. If $p$ does not fork over $A$, $|A| < \kappa (D)$, and $p$ is nonalgebraic, then there exists an indiscernible set $I$ over $A$ with $|I| \ge \kappa (D)$ such that $\Av (I / M) = p$.
\end{fact}
\begin{proof}
  This follows from \cite[3.9]{hs-independence}. We have to check that $p \rest A$ has unboundedly-many realizations, but this is easy using the extension-existence property of forking (Fact \ref{forking-facts}) and the assumption that $p$ is nonalgebraic.
\end{proof}

We can conclude:

\begin{thm}\label{chainsat}
  Let $\lambda \ge |T|$. Assume that $D$ is stable in some $\mu \le \lambda$. Let $\delta$ be a limit ordinal with $\cf{\delta} \ge \kappa (D)$ and let $\seq{M_i : i < \delta}$ be an increasing sequence of $(D, \lambda)$-homogeneous models. Then $\bigcup_{i < \delta} M_i$ is $(D, \lambda)$-homogeneous.
\end{thm}
\begin{proof}
  By cofinality consideration, we can assume without loss of generality that $\delta = \cf{\delta}$ and $\lambda > \delta$. Also without loss of generality, $\lambda$ is regular. Let $M_\delta := \bigcup_{i < \delta} M_i$. Let $A \subseteq |M_\delta|$ have size less than $\lambda$ and let $p \in S_D^{<\omega} (A)$. Let $q \in S_D^{<\omega} (M_\delta)$ be an extension of $p$ and assume for sake of contradiction that $p$ is not realized in $M_\delta$. By the moreover part of local character (Fact \ref{forking-facts}), there exists $i < \delta$ and $B \subseteq |M_i|$ such that $|B| < \kappa (D)$ and $q$ does not fork over $B$. By making $A$ slightly bigger we can assume without loss of generality that $B \subseteq A$.

  Since $p$ is not realized in $M_\delta$, $q$ is nonalgebraic. By Fact \ref{indisc-existence}, there exists an indiscernible set $I$ over $B$ with $\Av (I / M_\delta) = q$. Enlarging $I$ if necessary, $|I| = \lambda$. Since $M_{i + 1}$ is $(D, \lambda)$-homogeneous, we can assume without loss of generality that $I \subseteq |M_{i + 1}|$. By Fact \ref{indisc-fact} used $|A|$-many times (recall $|A| < \lambda$), there exists $I_0 \subseteq I$ with $|I_0| = \lambda$ and $I_0$ indiscernible over $A$. Then $\Av (I_0 / M_\delta) = \Av (I / M_\delta) = q$ so $p = \Av (I_0 / A)$. By definition of average, if $\phi (\bx, \ba) \in p$, there exists $\bb \in I_0$ such that $\models \phi[\bb, \ba]$. By indiscernibility over $A$, this is true for any $\bb \in I_0$, hence any element of $I_0$ realizes $p$.
\end{proof}
\begin{remark}
  When $\lambda > |T|$ and $\kappa (D) = \aleph_0$, Theorem \ref{chainsat} generalizes to superstable tame AECs with amalgamation (see \cite{bv-sat-v3} and the more recent \cite[6.10]{vv-symmetry-transfer-v3}). We do not know whether there is a generalization of Theorem \ref{chainsat} to AECs when $\lambda = \LS (\K)$ (see also \cite[Question 6.12]{vv-symmetry-transfer-v3}).
\end{remark}

In homogeneous model theory, superstability follows from categoricity:

\begin{lem}\label{categ-superstab}
If a homogeneous diagram $D$ in a first-order theory $T$ is categorical in a $\lambda > |T|$, then $\kappa (D) = \aleph_0$.
\end{lem}
\begin{proof}
  By Fact \ref{good-frame-fact} (applied to $\K := \K_D$, recall Proposition \ref{kd-aec}), $D$ is stable in all cardinals and in particular in $\mu := \aleph_\omega (|T|)$. Since $\mu^{\aleph_0} > \mu$, Fact \ref{stab-spectrum} gives $\kappa (D) = \aleph_0$.
\end{proof}

Note that Lemma \ref{categ-superstab} was known when $\lambda \neq \aleph_\omega (|T|)$ (see \cite[Theorem 3]{hyt-categ-homog}). The case $\lambda = \aleph_\omega (|T|)$ is new (in fact, once Lemma \ref{categ-superstab} is proven for $\lambda = \aleph_\omega (|T|)$, Hyttinen's argument for transferring categoricity \cite[14.(ii)]{hyt-categ-homog} goes through).

The referee asked if Lemma \ref{categ-superstab} had an easier proof using tools specific to homogeneous model theory. An easy proof of Lemma \ref{categ-superstab} when $\lambda \neq \aleph_{\omega} (|T|)$ runs as follows: By a standard Ehrenfeucht-Mostowski (EM) model argument of Morley (see for example \cite[8.21]{baldwinbook09}), $D$ is stable in every $\mu \in [|T|, \lambda)$. If $\lambda > \aleph_\omega (|T|)$, then $D$ is stable in $\mu := \aleph_\omega (|T|)$ and $\mu^{\aleph_0} > \mu$ so by the stability spectrum theorem (Fact \ref{stab-spectrum}), we must have that $\kappa (D) = \aleph_0$. If $\lambda < \aleph_\omega (|T|)$, $\lambda$ is a successor and we can use other EM model tricks. Only the case $\lambda = \aleph_\omega (|T|)$ remains but to deal with it, we are not aware of any tools specific to the homogeneous setup. The proof above is in effect an application of a result of Shelah and Villaveces (see \cite[2.2.1]{shvi635} and the recent exposition \cite{shvi-notes-v3-toappear}) and an upward stability transfer of the author \cite[5.6]{ss-tame-jsl}.

We can conclude with a proof of Theorem \ref{abstract-thm-2} from the abstract. When $\lambda = |T|^+$, we could appeal to \cite{tamenessthree} but prefer to prove a more general statement using primes: 

\begin{thm}\label{homog-precise}
  If a homogeneous diagram $D$ in a first-order theory $T$ is categorical in a $\lambda > |T|$, then the class $\Ksatp{|T|}_D$ of its $(D,|T|)$-homogeneous models is categorical in all $\lambda' \ge |T|$. In particular, if $D$ is also categorical in $|T|$, then $D$ is categorical in all $\lambda' \ge |T|$
\end{thm}
\begin{proof}
  Let $\K := \K_D$ be the class of $D$-models of $T$. By Proposition \ref{kd-aec}, $\K$ is a $\LS (\K)$-tame AEC (where $\LS (\K) = |T|$) with amalgamation and no maximal models. Furthermore $\K$ is categorical in $\lambda$. By Lemma \ref{categ-superstab}, $\kappa (D) = \aleph_0$. By Theorem \ref{chainsat}, the union of any increasing chain of $(D, |T|)$-homogeneous models is $(D, |T|)$-homogeneous. Moreover, there is a unique $(D, |T|)$-homogeneous model of cardinality $|T|$ (see e.g.\ \cite[5.9]{grle-homog}). So we get that:

  \begin{enumerate}
    \item $\Ksatp{|T|}_D$ is an AEC with $\LS (\Ksatp{|T|}_D) = \LS (\K)$.
    \item $\Ksatp{|T|}_D$ has amalgamation, no maximal models, and is $\LS (\K)$-tame.
    \item $\Ksatp{|T|}_D$ is categorical in $\LS (\K)$ and $\lambda$.
  \end{enumerate}

  Thus the last sentence in the statement of the theorem follows from uniqueness of homogeneous models. Let us prove the first. By Fact \ref{forking-facts}, nonforking induces a type-full good $|T|$-frame on the class $(\Ksatp{|T|}_D)_{|T|}$. By Fact \ref{homog-primes}, $\Ksatp{|T|}_D$ has primes. Now apply Theorem \ref{unidim-equiv}.
\end{proof}

\begin{thm}\label{abstract-thm-2-proof}
  If a homogeneous diagram $D$ in a first-order theory $T$ is categorical in a $\lambda > |T|$, then it is categorical in all $\lambda' \ge \min (\lambda, \hanf{|T|})$.
\end{thm}
\begin{proof}
  By Theorem \ref{homog-precise}, $\Ksatp{|T|}_D$ is categorical in all $\lambda' \ge |T|$. In particular by categoricity in $\lambda$, $\left(\Ksatp{|T|}_D\right)_{\ge \lambda} = (\K_D)_{\ge \lambda}$, so $\K_D$ is categorical in all $\lambda' \ge \lambda$. To see that $\K_D$ is categorical in all $\lambda' \ge \hanf{|T|}$, use Theorem \ref{abstract-thm-2-proof-0} (or just directly Fact \ref{omit-type}).
\end{proof}

\bibliographystyle{amsalpha}
\bibliography{categ-without-shortness}

\providecommand{\bysame}{\leavevmode\hbox to3em{\hrulefill}\thinspace}
\providecommand{\MR}{\relax\ifhmode\unskip\space\fi MR }
\providecommand{\MRhref}[2]{%
  \href{http://www.ams.org/mathscinet-getitem?mr=#1}{#2}
}
\providecommand{\href}[2]{#2}
\begin{thebibliography}{BKV06}

\bibitem[Bal09]{baldwinbook09}
John~T. Baldwin, \emph{Categoricity}, University Lecture Series, vol.~50,
  American Mathematical Society, 2009.

\bibitem[BG]{bg-v11-toappear}
Will Boney and Rami Grossberg, \emph{Forking in short and tame {A}{E}{C}s},
  Annals of Pure and Applied Logic, To appear. URL:
  \url{http://arxiv.org/abs/1306.6562v11}.

\bibitem[BGVV]{shvi-notes-v3-toappear}
Will Boney, Rami Grossberg, Monica VanDieren, and Sebastien Vasey,
  \emph{Superstability from categoricity in abstract elementary classes},
  Annals of Pure and Applied Logic, To appear. URL:
  \url{http://arxiv.org/abs/1609.07101v3}.

\bibitem[BKV06]{b-k-vd-spectrum}
John~T. Baldwin, David Kueker, and Monica VanDieren, \emph{Upward stability
  transfer for tame abstract elementary classes}, Notre Dame Journal of Formal
  Logic \textbf{47} (2006), no.~2, 291--298.

\bibitem[Bon]{tameness-groups}
Will Boney, \emph{Tameness in groups and similar {A}{E}{C}s}, In preparation.

\bibitem[Bon14]{tamelc-jsl}
\bysame, \emph{Tameness from large cardinal axioms}, The Journal of Symbolic
  Logic \textbf{79} (2014), no.~4, 1092--1119.

\bibitem[BS08]{non-locality}
John~T. Baldwin and Saharon Shelah, \emph{Examples of non-locality}, The
  Journal of Symbolic Logic \textbf{73} (2008), 765--782.

\bibitem[BVa]{bv-sat-v3}
Will Boney and Sebastien Vasey, \emph{Chains of saturated models in
  {A}{E}{C}s}, Preprint. URL: \url{http://arxiv.org/abs/1503.08781v3}.

\bibitem[BVb]{counterexample-frame-v2}
\bysame, \emph{Good frames in the {H}art-{S}helah example}, Preprint. URL:
  \url{http://arxiv.org/abs/1607.03885v2}.

\bibitem[BVc]{bv-survey-v4-toappear}
\bysame, \emph{A survey on tame abstract elementary classes}, Beyond first
  order model theory (Jos{\'e} Iovino, ed.), CRC Press, To appear. URL:
  \url{http://arxiv.org/abs/1512.00060v4}.

\bibitem[GL02]{grle-homog}
Rami Grossberg and Olivier Lessmann, \emph{Shelah's stability spectrum and
  homogeneity spectrum in finite diagrams}, Archive for Mathematical Logic
  \textbf{41} (2002), no.~1, 1--31.

\bibitem[Gro02]{grossberg2002}
Rami Grossberg, \emph{Classification theory for abstract elementary classes},
  Contemporary Mathematics \textbf{302} (2002), 165--204.

\bibitem[GV]{gv-superstability-v4}
Rami Grossberg and Sebastien Vasey, \emph{Equivalent definitions of
  superstability in tame abstract elementary classes}, Preprint. URL:
  \url{http://arxiv.org/abs/1507.04223v4}.

\bibitem[GV06a]{tamenessthree}
Rami Grossberg and Monica VanDieren, \emph{Categoricity from one successor
  cardinal in tame abstract elementary classes}, Journal of Mathematical Logic
  \textbf{6} (2006), no.~2, 181--201.

\bibitem[GV06b]{tamenessone}
\bysame, \emph{Galois-stability for tame abstract elementary classes}, Journal
  of Mathematical Logic \textbf{6} (2006), no.~1, 25--49.

\bibitem[GV06c]{tamenesstwo}
\bysame, \emph{Shelah's categoricity conjecture from a successor for tame
  abstract elementary classes}, The Journal of Symbolic Logic \textbf{71}
  (2006), no.~2, 553--568.

\bibitem[Har75]{harniksat}
Victor Harnik, \emph{On the existence of saturated models of stable theories},
  Proceedings of the American Mathematical Society \textbf{52} (1975),
  361--367.

\bibitem[HK06]{finitary-aec}
Tapani Hyttinen and Meeri Kes{\"a}l{\"a}, \emph{Independence in finitary
  abstract elementary classes}, Annals of Pure and Applied Logic \textbf{143}
  (2006), 103--138.

\bibitem[HS00]{hs-independence}
Tapani Hyttinen and Saharon Shelah, \emph{Strong splitting in stable
  homogeneous models}, Annals of Pure and Applied Logic \textbf{103} (2000),
  201--228.

\bibitem[Hyt98]{hyt-categ-homog}
Tapani Hyttinen, \emph{Generalizing {M}orley's theorem}, Mathematical Logic
  Quarterly \textbf{44} (1998), 176--184.

\bibitem[Les00]{lessmann2000}
Olivier Lessmann, \emph{Ranks and pregeometries in finite diagrams}, Annals of
  Pure and Applied Logic \textbf{106} (2000), 49--83.

\bibitem[Les05]{lessmann-upward-transfer}
\bysame, \emph{Upward categoricity from a successor cardinal for tame abstract
  elementary classes with amalgamation}, The Journal of Symbolic Logic
  \textbf{70} (2005), no.~2, 639--660.

\bibitem[Lie11]{lieberman2011}
Michael~J. Lieberman, \emph{A topology for {G}alois types in abstract
  elementary classes}, Mathematical Logic Quarterly \textbf{57} (2011), no.~2,
  204--216.

\bibitem[She70]{sh3}
Saharon Shelah, \emph{Finite diagrams stable in power}, Annals of Mathematical
  Logic \textbf{2} (1970), no.~1, 69--118.

\bibitem[She75]{sh54}
\bysame, \emph{The lazy model theoretician's guide to stability}, Logique et
  Analyse \textbf{18} (1975), 241--308.

\bibitem[She83a]{sh87a}
\bysame, \emph{Classification theory for non-elementary classes {I}: The number
  of uncountable models of $\psi \in {L}_{\omega_1, \omega}$. {P}art {A}},
  Israel Journal of Mathematics \textbf{46} (1983), no.~3, 214--240.

\bibitem[She83b]{sh87b}
\bysame, \emph{Classification theory for non-elementary classes {I}: The number
  of uncountable models of $\psi \in {L}_{\omega_1, \omega}$. {P}art {B}},
  Israel Journal of Mathematics \textbf{46} (1983), no.~4, 241--273.

\bibitem[She87]{sh300-orig}
\bysame, \emph{Universal classes}, Classification theory (Chicago, IL, 1985)
  (John~T. Baldwin, ed.), Lecture Notes in Mathematics, vol. 1292,
  Springer-Verlag, 1987, pp.~264--418.

\bibitem[She99]{sh394}
\bysame, \emph{Categoricity for abstract classes with amalgamation}, Annals of
  Pure and Applied Logic \textbf{98} (1999), no.~1, 261--294.

\bibitem[She09]{shelahaecbook}
\bysame, \emph{Classification theory for abstract elementary classes}, Studies
  in Logic: Mathematical logic and foundations, vol.~18, College Publications,
  2009.

\bibitem[SV99]{shvi635}
Saharon Shelah and Andr{\'e}s Villaveces, \emph{Toward categoricity for classes
  with no maximal models}, Annals of Pure and Applied Logic \textbf{97} (1999),
  1--25.

\bibitem[Vasa]{ap-universal-v10}
Sebastien Vasey, \emph{Shelah's eventual categoricity conjecture in universal
  classes: part {I}}, Preprint. URL: \url{http://arxiv.org/abs/1506.07024v10}.

\bibitem[Vasb]{categ-universal-2-v3-toappear}
\bysame, \emph{Shelah's eventual categoricity conjecture in universal classes:
  part {I}{I}}, Selecta Mathematica, To appear. URL:
  \url{http://arxiv.org/abs/1602.02633v3}.

\bibitem[Vas16a]{indep-aec-apal}
\bysame, \emph{Building independence relations in abstract elementary classes},
  Annals of Pure and Applied Logic \textbf{167} (2016), no.~11, 1029--1092.

\bibitem[Vas16b]{ss-tame-jsl}
\bysame, \emph{Forking and superstability in tame {A}{E}{C}s}, The Journal of
  Symbolic Logic \textbf{81} (2016), no.~1, 357--383.

\bibitem[Vas17]{downward-categ-tame-apal}
\bysame, \emph{Downward categoricity from a successor inside a good frame},
  Annals of Pure and Applied Logic \textbf{168} (2017), no.~3, 651--692.

\bibitem[VV]{vv-symmetry-transfer-v3}
Monica VanDieren and Sebastien Vasey, \emph{Symmetry in abstract elementary
  classes with amalgamation}, Preprint. URL:
  \url{http://arxiv.org/abs/1508.03252v3}.

\end{thebibliography}

\end{document}